\documentclass{article}

\usepackage{etex}

\usepackage{latexsym, a4wide}
\usepackage{amsmath, rotating}
\usepackage{amsmath, rotating, color}
\usepackage{amsfonts,amssymb, amsthm}
\usepackage{amstext,amsthm,amsmath}
\usepackage{color}
\usepackage[normalem]{ulem}
\usepackage{amssymb,mathtools}
\usepackage{pictex, natbib}
\usepackage{dsfont} % \1
\newcommand\unnumberedfootnote[1]{ %
        \let\temp=\thefootnote %
        \renewcommand{\thefootnote}{}%
        \footnote{#1}%
        \let\thefootnote=\temp%
        \addtocounter{footnote}{-1}}

\newcommand{\1}{\mathds{1}}
\newcommand{\E}{\mathbb E}
\newcommand{\Pw}{\mathbb P}

\newcommand{\cL}{\mathcal L}

\newcommand{\N}{\mathbb N}
\newcommand{\R}{\mathbb R}

\newtheorem{theorem}{Theorem}
\newtheorem{proposition}{Proposition}[section]
\newtheorem{lemma}[proposition]{Lemma}

\newtheorem{definition}[proposition]{Definition}
\theoremstyle{definition}
\newtheorem{remark}[proposition]{Remark}

\numberwithin{equation}{section}
\usepackage[usenames,dvipsnames]{xcolor}
\usepackage{tikz}
\usetikzlibrary{
  graphs,
  graphs.standard
}
\usepackage{tikz-qtree}

\begin{document}

\title{\LARGE 
The partial duplication random graph with edge deletion}

\thispagestyle{empty}

\author{{\sc by  Felix Hermann and Peter Pfaffelhuber} \\[2ex]
  \emph{Albert-Ludwigs University Freiburg} } \date{March 28th}

\maketitle
\unnumberedfootnote{\emph{AMS 2000 subject classification.} {\tt
    05C80} (Primary) {\tt , 60K35} (Secondary).}

\unnumberedfootnote{\emph{Keywords and phrases.} Random graph; degree
  distribution; cliques}

\vspace*{-5ex}

\begin{abstract}
  \noindent
  We study a random graph model in continuous time. Each vertex is
  partially copied with the same rate, i.e.\ an existing vertex is
  copied and every edge leading to the copied vertex is copied with
  independent probability $p$. In addition, every edge is deleted at
  constant rate, a mechanism which extends previous partial
  duplication models. In this model, we obtain results on the degree
  distribution, which shows a phase transition such that either -- if
  $p$ is small enough -- the frequency of isolated vertices converges
  to~1, or there is a positive fraction of vertices with unbounded
  degree. We derive results on the degrees of the initial vertices as
  well as on the sub-graph of non-isolated vertices. In particular, we
  obtain expressions for the number of star-like subgraphs and
  cliques.
\end{abstract}

\section{Introduction}
Various random graph models have been studied in the last decades.
Frequently, such models try to mimic the behavior of social networks
(see e.g.\ \citealp{CooperFrieze2003} and \citealp{BarabasiEtAl2002})
or interactions within biological networks (see e.g.\
\citealp{Wagner2001}, \citealp{Albert2005} and \citealp{Jeong2000}).
For a general introduction to random graphs see the monographs
\cite{Durrett}, \cite{Remco} and references therein.

In this paper, we study and extend a model introduced in
\cite{Bhan2002}, \cite{ChungEtAl2003}, \cite{PastorSatorras2003},
\cite{ChungEtAl2003}, \cite{BebekEtAl2006}, \cite{BebekEtAl2006ip},
\cite{HerPf1} and \cite{Jordan2017}. Here, a vertex models a protein
and an edge denotes some form of interaction. Within the genome, the
DNA encoding for a protein can be duplicated (which in fact is a long
evolutionary process), such that the interactions of the copied
protein are partially inherited to the copy; see \cite{Ohno1970} for
some more biological explanations. Within the random graph model, a
vertex is $p$-copied, i.e.\ a new vertex is introduced and every edge
of the parent vertex is independently copied with the same probability
$p$.  We will extend this model by assuming mutation (at some constant
rate), which can destroy interactions, leading to the loss of edges in
the random graph at constant rate.

\noindent
In the model without edge deletion, which we will call \emph{pure
  partial duplication model}, \cite{HerPf1} have determined a critical
parameter $p^\ast\approx0.567143$, the unique solution of $pe^p=1$,
below which approximately all vertices are isolated. Moreover, almost
sure asymptotics and limit results for the number of $k$-cliques and
$k$-stars in the random graph as well as for the degree of a fixed
node were obtained. Recently, \cite{Jordan2017} has shown that for
$p<e^{-1}$ %(where $e^{-1}<p^\ast$)
the degree distribution of the connected component, i.e. of the
subgraph of non-isolated vertices, has a limit with tail behavior
close to a power-law with exponent $\beta$ solving
$\beta-3+p^{\beta-2}=0$ (cf.\ \citealp[Theorem 1(c)]{Jordan2017}).

In the model with edge deletion, we will add results on the degree
distribution of the full graph (see
Theorem~\ref{thm:deg-dis-whole-graph}), the sub-graph of non-isolated
vertices (see Proposition~\ref{prop:limitOfI_x}) as well as the number
of star-like graphs, cliques, and degrees of initial vertices (see
Theorem~\ref{thm:functionals}).  We will see that the degree
distribution of this model is closely related to a branching process
as in \cite{Jordan2017}, but now with death-rate $\delta$, the rate of
edge deletion. In turn, such branching processes can be studied by
using piecewise-deterministic Markov jump processes, a tool which we
will introduce in Lemma~\ref{lem:dual-p-jump-process}; see also
Section~\ref{ss31} for the connection to branching processes. These
insights allow us to transfer results on branching processes to the
degree distribution, generalizing the results of Theorem 2.7 in
\cite{HerPf1} to the model with edge deletion; see
Section~\ref{S:proof1}. Here, we derive a phase transition such that
if $p$ is small, the fraction of vertices with positive degree
vanishes, while for larger $p$, vertices are either isolated (i.e.\
have vanishing degree) or have unbounded degree. In Section
\ref{S:proof2}, we prove Theorem~\ref{thm:functionals} and derive
almost sure asymptotics and limit results for binomial moments,
cliques and the degree of a fixed node mainly by applying martingale
theory, generalizing Theorems~2.9 and~2.14 in \cite{HerPf1}.

\section{Model and main results}
\label{S:model}

\begin{definition}[Partial duplication graph process with edge
  deletion]\label{def:PDwEdgeDel}
  Let $p\in[0,1]$, $\delta\geq0$ and $G_0=(V_0,E_0)$ be a deterministic undirected graph
  without loops with vertex set $V_0=\{v_1,\ldots,v_{|V_0|}\}$ and non-empty edge set $E_0$.
  Let $PD(p,\delta)=(G_t)_{t\geq0}$ be the continuous-time graph-valued Markov process
  starting in $G_0$ and evolving in the following way:
  \begin{itemize}
  \item[--] Every node $v\in V_t$ is $p$-partially duplicated (or
    \emph{$p$-copied} for short) at rate $\kappa_t:=(|V_t|+1)/|V_t|$,
    i.e.\ a new node $v_{|V_t|+1}$ is added and for each $w\in V_t$
    with $(v,w)\in E_t$, $v_{|V_t|+1}$ is connected to $w$
    independently with probability $p$.
    \item[--]
      Every edge in $E_t$ is removed at rate $\delta$.
  \end{itemize}
  Then, $PD(p,\delta)$ is a \emph{partial duplication graph process
    with edge deletion} with \emph{initial graph} $G_0$,
  \emph{edge-retaining probability} $p$ and \emph{deletion rate}
  $\delta$.
  Within $PD(p,\delta)$, we define the following quantities:
  \begin{enumerate}
  \item Let $D_i(t):=\deg_{G_t}(v_i)\cdot\1_{\{i\leq|V_t|\}}$ be the
    \emph{degree of $v_i$}, i.e. the number of its neighbors, at time
    $t$.
  \item Let $F(t):=(F_k(t))_{k=0,1,2,...}$ with
    $F_k(t):=|\{1\leq i\leq |V_t|:D_i(t)=k\}|/|V_t|$ for $k=0,1,2,...$
    be the \emph{degree distribution} at time $t$.  Furthermore, let
    $F_+(t):=1-F_0(t)$ be the proportion of vertices of positive
    degree.
  \item For $k=1,2,...$ let
    $B_k(t):=\sum_{\ell\geq k}\binom \ell k F_\ell(t)$ be the
    \emph{$k$th binomial moment} of the degree distribution. Note that
    $B_k(t)$ is the expected number of subgraphs of $V_t$ with $k$
    nodes, all of which share one randomly chosen center, i.e.\ the
    expected number of star-like trees with $k$ leaves.
  \item For $k=1,2,...$ let $C_k(t)$ be the \emph{number of $k$-cliques}
    at time $t$, i.e. the number of complete sub-graphs of size $k$.
  \end{enumerate}
\end{definition}

\begin{remark}
  \begin{enumerate}
  \item Note that $PD(p, \delta)$ is a time-continuous duplication
    model and duplication events (total rate $|V_t|+1)$ at time $t$
    and the deletion events (total rate $\delta |E_t|$ at time $t$)
    depend on different statistics. An alternative, discrete model
    would add a node or delete an edge with some fixed
    probability. However, such a model would behave differently from
    the model above.
  \item We choose the duplication rate $\kappa_t := (|V_t|+1) / |V_t|$
    in order to get a closed recurrence relation for the degree
    distribution; see Lemma \ref{lem:duality}. Alternatively, we would
    set $\widetilde \kappa_t:=1$, i.e.\ all vertices are copied at
    unit rate. Since $V_t \sim e^t$ (see
    Lemma~\ref{lem:yule-process}), and
    $\int_0^t \kappa_s - \widetilde\kappa_s ds \sim \int_0^t e^{-s} ds
    \xrightarrow{t\to\infty} 1 \ll \int_0^t \kappa_s ds$, we expect
    the random graph with our choice of $\kappa_t$ to behave the same
    qualitatively for $t\to\infty$, while differences can occur
    quantitatively.
  \end{enumerate}
\end{remark}

\noindent
In order to formulate our results, we need an auxiliary process, which
is connected to $PD(p,\delta)$. It will appear below in
Theorem~\ref{thm:deg-dis-whole-graph}.1 and in
Proposition~\ref{prop:limitOfI_x}. The proof of the following Lemma is
found in Section~\ref{ss32}.

\begin{lemma}[Connection of $PD(p,\delta)$ and a
  piecewise-deterministic Markov
  process]\label{lem:dual-p-jump-process}
  Let $X = (X_t)_{t\geq 0}$ be a Markov process on $[0,1]$ jumping at
  rate 1 from $X_t = x$ to $px$, in between jumps evolving according
  to $\dot X_t=pX_t(1-X_t)-\delta X_t$.  Furthermore, let
  \begin{align}
    \label{eq:H}
    H_x(t)
    := \sum_{k=0}^\infty(1-x)^kF_k(t),
  \end{align}
  i.e.\ the probability generating function of the degree distribution
  at time $t$.  Then, for all $t\geq0$ and $x\in[0,1]$, writing
  $\mathbb E_x[.] := \mathbb E[.|X_0=x]$,
  \begin{align}
    \label{eq:duality}
    \E[H_x(t)] = \E_x[H_{X_t}(0)] = \sum_{k=0}^\infty F_k(0) \cdot
    \E_x[(1-X_t)^k].
  \end{align}
\end{lemma}

\begin{theorem}[Limit of the degree distribution]\label{thm:deg-dis-whole-graph}
  Let $p\in(0,1)$ and $\delta\geq0$.
  \begin{enumerate}
  \item If $\delta\geq p-p\log\frac1p$,
    $F(t)\xrightarrow{t\to\infty}(1,0,0,\ldots)$ almost surely with
    (recall $F_+ = 1-F_0$, and writing $a_t \sim b_t$ iff
    $a_t/b_t\xrightarrow{t\to\infty} 1$)
    \begin{align*}
      \E[F_+(t)] \sim ce^{-t(1 + \delta - 2p)},
    \end{align*}
    where $X$ is as in
    Lemma~\ref{lem:dual-p-jump-process}
    and
    $$ c = \exp\Big(-p\int_0^\infty \frac{\mathbb E_1[X_s^2]}{\mathbb E_1[X_s]}ds\Big) \in (0,B_1(0)).$$
  \item If $p-p\log\frac1p \geq \delta\geq p-\log\frac1p$,
    $F(t)\xrightarrow{t\to\infty}(1,0,0,\ldots)$ almost
    surely with
    \begin{align*}
      -\tfrac1t\log\E[F_+(t)]
      \xrightarrow{t\to\infty}
      1 - \tfrac1\gamma(1+\log\gamma),
    \end{align*}
    where $\gamma=\log\frac1p/(p-\delta)$.
    \item
      If $\delta<p-\log\frac1p$, $F(t)\xrightarrow{t\to\infty}(F_0,0,0,\ldots)$ almost surely,
      where $F_0$ is non-deterministic and
      \begin{align*}
        \E[F_0]
        &= 1 - \Big(1-\frac\delta p-\frac1p\log\frac1p\Big)
          \sum_{k=1}^\infty B_k(0)(-1)^{k-1}
          \prod_{\ell=1}^{k-1}\Big(1-\frac\delta p-\frac{1-p^\ell}{p\ell}\Big),
      \end{align*}
      where we define $\prod_\emptyset := 1$.
    \end{enumerate}
\end{theorem}

\begin{remark}[Interpretations]\label{rem:thm:deg-dis-whole-graph}
  \begin{enumerate}
  \item Clearly, the quantity $F_0$ is increasing in $\delta$ and
    decreasing in $p$ (and $F_+$ is decreasing in $\delta$ and
    increasing in $p$). See also the illustrations of
    Theorem~\ref{thm:deg-dis-whole-graph} in Figure~\ref{fig1}. We
    note that for $\delta=0$, the three cases can be distinguished
    using $p^\ast$, the solution of $pe^p=1$ (or
    $\frac 1p \log \frac 1p = 1$). The three cases are then
    $p\leq 1/e$,
    $ 1/e\leq p \leq p^\ast$ and $p> p^\ast$.\\
    For 3., we will see in the proof that the right hand side is the
    hitting probability of a stochastic process, and in particular is
    in $(0,1)$. This interpretation shows that $0<\mathbb E[F_0]<1$ as
    long as the initial graph is not trivial (i.e.\ $F_0(0)<1$).
  \item The asymptotics given in case $\delta\geq p-p\log\frac1p$ is
    more exact than the one given for
    $p-p\log\frac1p \geq \delta\geq p - \log \frac 1p$ (in the sense
    that $-\tfrac 1t \log \mathbb E[F_+(t)] \sim 1 + \delta - 2p$ is a
    consequence of 1. of the above Theorem). The reason is that we can
    give a formula for $c$ in this case, which does not carry over to
    2.; see the proof of Proposition~\ref{prop:limitOfI_x}.
  \end{enumerate}
\end{remark}

\begin{figure}[htb]
  \begin{center}
    \begin{tikzpicture}[xscale=6.4,yscale=3.2]
      \draw (.5,1.3) node {(A)};
      \draw [thick,->] (0,0)--(1.1,0);
      \draw [thick,->] (0,0)--(0,1.1);
      \draw [dashed] (1,0)--(1,1.1);
      \draw [dashed] (0,1)--(1,1);
      \draw (1.1,-.1) node {$p$};
      \draw (1,-.1) node {$1$};
      \draw (.57,-.1) node {$p^\ast$};
      \draw (.37,-.1) node {$1/e$};
      \draw (-.05,1) node {$1$};
      \draw (-.05,1.2) node {$\delta$};
      
      \draw [domain=0.568:1] plot ({\x}, {\x-ln(1/\x)});
      \draw [domain=0.37:1] plot ({\x}, {\x-\x*ln(1/\x)});
      
      \draw (.8,.1) node {\scriptsize $\mathbb E[F_0(\infty)] < 1$};
      \draw (.3,.6) node {\scriptsize $\mathbb E[F_+(\infty)] \sim e^{-t(1+\delta-2p)}$};
      \draw (.32,.32) node {\scriptsize $\mathbb E[F_+(\infty)] \! \sim \! e^{-t(1 - \tfrac 1\gamma(1+\log \gamma))}$};
      
      \draw [->] (.39,.26)--(.55, .1);
    \end{tikzpicture}
    \begin{tikzpicture}[xscale=5,yscale=1.6]
      \draw (.5,2.6) node {(B)};
      \draw [thick,->] (0,0)--(1.1,0);
      \draw [thick,->] (0,0)--(0,2.2);
      \draw [dashed] (1,0)--(1,2.1);
      \draw (1.1,-.1) node {$p$};
      \draw (1,-.1) node {$1$};
      \draw (.57,-.15) node {$p^\ast$};
      \draw (.37,-.15) node {$1/e$};
      \draw (0,2.3) node {$-\tfrac 1t \log \mathbb E[F_+(t)]$};
      
      \draw [thick,domain=0:0.3678794] plot ({\x}, {1-2*\x});
      \draw [thick,domain=0.3678794:0.57] plot ({\x}, {1+\x/ln(\x)*(1 + ln(ln(1/\x)) - ln(\x))});
      \draw (.15,.5) node {\scriptsize $\delta=0$};
      
      \draw [thick,domain=0:1] plot ({\x}, {2-2*\x});
      \draw (.6,1) node {\scriptsize $\delta=1$};
      
      \draw [thick,domain=0:0.729843] plot ({\x}, {1+.5-2*\x});
      \draw [thick,domain=0.729846:0.766247] plot ({\x}, {1+(\x-.5)/ln(\x)*(1 + ln(ln(1/\x)) - ln(\x-.5))});
      \draw (.375,.75) node {\scriptsize $\delta=0.5$};
      
      \draw [dashed,domain=0.37:1] plot ({\x}, {1+\x - \x*ln(1/\x) - 2*\x});
      
      \draw [dashed] (.37,0) -- (0.37, 0.2);
      \draw (.75,.25) node {\scriptsize $1 - p\log \tfrac 1p - p$};
      \draw [->] (.75,.2)--(.6, .1);
    \end{tikzpicture}
  \end{center}
  \caption{\label{fig1}}Illustration of
  Theorem~\ref{thm:deg-dis-whole-graph}. In (A), the three cases are
  shown in the $p-\delta$-plane. In (B), we draw the different
  exponential rates of decrease in $\mathbb E[F_+]$.
\end{figure}

~

\noindent
In the case $\delta\geq p - \log \frac 1p$, the frequency of isolated
vertices converges to~1. Hence, it is interesting to study the (degree
distribution of the) sub-graph of non-isolated vertices.  In order to
do so, note that $F_+(t) = 1 - H_1(t)$, with $H$ from
\eqref{eq:H}. Also note that, if at some time $t$ a duplication event
is triggered, $1-H_p(t)$ denotes the probability that the new node is
not isolated. The next result gives asymptotics of the generating
function of the degree distribution of the sub-graph of non-isolated
vertices,
$$ \frac{\sum_{k=1}^{\infty} (1-x)^k \mathbb E[F_k(t)]}{\mathbb
  E[F_+(t)]} = \frac{\mathbb E[H_x(t) - H_1(t)]}{\mathbb E[1 -
  H_1(t)]} = 1 - \frac{\mathbb E[1 - H_x(t)]}{\mathbb E[1 -
  H_1(t)]}.$$

\begin{proposition}[Limit of degree distribution on the set of
  non-isolated vertices]\label{prop:limitOfI_x}
  Let $X$ be the process given in Lemma~\ref{lem:dual-p-jump-process}.
  \begin{enumerate}
    \item
      If $\delta>p-p\log\frac1p$, then $\E[1 - H_x(t)]\sim\E_x[X_t]\cdot B_1(0)$.
    \item
      If $\delta=p-p\log\frac1p$ and if the limits
      $\E_x[X_t^{k+1}]/\E_x[X_t^k]$ as $t\to\infty$ exist for
      all $x \in (0,1]$ and $k=1,2,...$, then
      $\E[1 - H_x(t)]\sim\E_x[X_t]\cdot B_1(0)$.
    \item If $p-\log\frac1p < \delta < p-p\log\frac1p$ and if the
      limits $\E_x[X_t^{k+1}]/\E_x[X_t^k]$ as $t\to\infty$ exist for
      all $x \in (0,1]$ and $k=1,2,...$, then
      \[
        \E[1-H_x(t)] \sim \E_x[X_t]\cdot
        \sum_{\ell=1}^\infty B_\ell(0)(-1)^{\ell+1}
        \prod_{k=1}^{\ell-1}\frac{\log\frac1{\gamma p^k} + \gamma p^k
          - 1}{k\gamma p},
      \]
      where $\gamma=\log\frac1p/(p-\delta)$.
  \end{enumerate}
  In both cases,
  \begin{align}
    \label{eq:P1}
    \displaystyle \frac{\E[1-H_x(t)]}{\E[F_+(t)]} \sim
    \frac{\E_x[X_t]}{\E_1[X_t]} \xrightarrow{t\to\infty}
    x\exp\Big(p\int_0^\infty\frac{\E_1[X_s^2]}{\E_1[X_s]}-\frac{\E_x[X_s^2]}{\E_x[X_s]}ds\Big).
  \end{align}
\end{proposition}

\begin{remark}[Interpretation]
  Since the right hand side of \eqref{eq:P1} is non-trivial, the
  Proposition shows that the degree distribution of the sub-graph of
  non-isolated vertices converges to some non-trivial
  distribution. However, for a closer analysis, more insight into the
  process $X$ would be necessary.
\end{remark}

\begin{remark}[Convergence of moments]\sloppy \label{rem:comvMom}
  In 2. and 3., we have to assume that the limits of
  $\E_x[X_t^{k+1}]/\E_x[X_t^k]$ exist. We conjecture that this
  assumption is not necessary since we can at least prove the weaker
  convergence in the sense of C\'esaro. Precisely, we can show that
  for $p-\log\frac1p<\delta<p-p\log\frac1p$, it holds that for
  $k=1,2,...$ and any $x\in (0,1]$
  \begin{align}\label{eq:ces}
    \frac 1t \int_0^t \frac{\E_x[X_s^{k+1}]}{\E_x[X_s^k]} ds \xrightarrow{t\to\infty} c_k(p,\delta)
    := \frac{\log\frac1{\gamma p^k} + \gamma p^k - 1}{k\gamma p}
    \in(0,1].
  \end{align}
  \\
  Indeed, since the moments of $X$ satisfy, for $k=1,2,...$,
  \begin{align}\label{eq:log-deriv-moments-of-X}
    \frac d{dt}\log\E_x[X_t^k]
    &= \frac1{\E_x[X_t^k]}\Big(p^k\E_x[X_t^k]-\E_x[X_t^k]
      + (p-\delta)k\E_x[X_t^k]-pk\E_x[X_t^{k+1}]\Big)\notag\\
    &= p^k+(p-\delta)k-1-\frac{pk\E_x[X_t^{k+1}]}{\E_x[X_t^k]}
  \end{align}
  and thus, integrating, dividing by $-t$, and using
  Corollary 2.4 of \cite{HerPf2},
  \begin{align*}
    \frac 1t \int_0^t \E_x[X_s^{k+1}]/\E_x[X_s^k] ds & = - \frac{1}{pk} \Big(\frac 1t(\log\E_x[X_t^k]-\log(x^k)) 
                                                       - p^k - (p-\delta)k +1\Big)
    \\ & \xrightarrow{t\to\infty} \frac{p - \delta + p^k - (1 + \log\gamma)/\gamma}{pk} = c_k(p,\delta).
  \end{align*}
  Clearly, if the limit as assumed in
  Proposition~\ref{prop:limitOfI_x} exists, it must hold that
  $\E_x[X_t^{k+1}]/\E_x[X_t^k] \xrightarrow{t\to\infty} c_k(p,
  \delta)$.\\
  Similarly, for the critical case $\delta=p-p\log\frac1p$, where $\gamma=1/p$,
  it follows from \eqref{eq:ces} for $k\geq2$
  \begin{align*}
    \frac1t\int_0^t \frac{\E_x[X_s^k]}{\E_x[X_s]} ds
      & \leq \frac1t\int_0^t \frac{\E_x[X_s^2]}{\E_x[X_s]} ds
        \xrightarrow{t\to\infty} c_1(p,\delta)
        = 0,
  \end{align*}
  such that $\E[X_t^k]=o(\E[X_t])$ if the limits exist.
\end{remark}

\noindent
We now investigate the limiting behavior of certain functionals of the
graph.

\begin{theorem}[Binomial moments, cliques and degrees]\label{thm:functionals}
  As $t\to\infty$, the following statements hold almost surely:
  \begin{enumerate}
    \item
      For $k=1,2,...$, $e^{t\beta_k}B_k(t)\to B_k(\infty)$, where
      $B_k(\infty)\in\cL^1$ and
      \[
        \beta_k =
        \begin{cases}
          1+\delta-2p, &\text{if }\delta\geq p-\frac{p(1-p^{k-1})}{k-1},\\
          1+\delta k-pk-p^k, &\text{otherwise.}
        \end{cases}
      \]
    \item
      For $k=2,3,...$, $\exp(-t(kp^{k-1}-\delta\genfrac(){0pt}{}k2))C_k(t)\to C_k(\infty)$,
      where $C_k(\infty)\in\cL^1$.
      \begin{enumerate}
        \item 
          If $C_k(0)>0$ and $\delta<2p^{k-1}/(k-1)$, the convergence also holds in $\cL^1$
          and $\Pw(C_k(\infty)>0)>0$.
        \item
          Otherwise, if $\delta\geq2p^{k-1}/(k-1)$, $C_k(t)=0$ for all $t\geq T^{C_k}$ for
          some finite random variable $T^{C_k}$ and $\Pw(C_k(\infty)=0)=1$.
      \end{enumerate}
    \item For $i\leq|V_0|$, $e^{-t(p-\delta)}D_i(t)\to D_i(\infty)$,
      where $D_i(\infty)\in\cL^1$.  
      \begin{enumerate}
      \item If $D_i(0)>0$ and $\delta < p$, the convergence also holds
        in $\cL^r$ for all $r\geq 1$ and
        \begin{align}
          \label{eq:212}
          \E[D_i(\infty)] = D_i(0)\Big( 1 + \frac{p}{|V_0|}\Big).
        \end{align}
      \item Otherwise, if $\delta\geq p$, $D_i(t)=0$ for all
        $t\geq T^{D_i}$ for some finite random variable $T^{D_i}$ and
        $\Pw(D_i(\infty)=0)=1$.
      \end{enumerate}
  \end{enumerate}
\end{theorem}

\begin{remark}[Interpretations]\label{lem:remInt}
  \begin{enumerate}
  \item For Theorem~\ref{thm:functionals}.1, we have
    $\beta_1 = 1 + \delta - 2p$ and for $\delta=0$, we have
    $\beta_k = 1 - pk - p^k$, $k=1,2,...$ In all cases, we can also
    write $\beta_k = (1 + \delta-2p)\wedge (1+\delta k - pk - p^k)$,
    which immediately shows that $\beta_k$ is continuous in $p$ and
    $\delta$. In addition, for $k=2,3,...$, we find
    $p\geq \frac{p(1-p^{k-1})}{k-1}$, i.e.\ we can choose
    $\delta\geq 0$ such that either of the two cases can in fact
    occur. Moreover, $\beta_k \leq \beta_{k-1}$, which can be seen as
    follows: First, note that
    $\frac{(1-p^{k-2})/(k-2)}{(1-p^{k-1})/(k-1)} = \frac{(1+\cdots
      +p^{k-3})/(k-2)}{(1+\cdots + p^{k-2})/(k-1)}\geq 1$. So, if
    $\delta \geq p - \frac{p(1-p^{k-1})}{k-1}$, both $\beta_{k-1}$ and
    $\beta_k$ do not depend on $k$ anyway. Then, if
    $ p - \frac{p(1-p^{k-1})}{k-1} \geq \delta \geq p -
    \frac{p(1-p^{k-2})}{k-2}$, we have that
    $$ \beta_{k-1} - \beta_k = 1 + \delta - 2p - \min(1 + \delta - 2p, 1 + \delta k - pk -p^k)\geq 0.$$
    Finally, for $p - \frac{p(1-p^{k-2})}{k-2} \geq \delta$, we have
    \begin{align*}
      \beta_{k-1} - \beta_k  & =
                               p - \delta - p^{k-1} + p^k \geq \frac{p(1-p^{k-1})}{k-1} -p(1-p)p^{k-2}
      \\ & = p(1-p)\Big( \frac{1 + \cdots + p^{k-2}}{k-1} - p^{k-2}\Big)\geq 0.
    \end{align*}
    The fact that $\beta_{k-1} \geq \beta_k$ implies that there are much
    less star-like subgraphs with $k-1$ leaves than star-like subgraphs
    with $k$ leaves, $k=2,3,...$ This can only be explained by nodes
    with high degree.
  \item Noting that $B_1(t) = 2|V(t)| \cdot C_2(t)$ and
    $ \frac 1t \log |V(t)| \xrightarrow{t\to\infty} 1$, we see that
    the results in 1.\ and 2.\ imply the same growth rate for the
    number of edges.
  \item Interestingly, we find that $\delta\geq p$ implies that all
    vertices of the initial graph will eventually be isolated (i.e.\
    have degree~0). However, the total number of edges, denoted by
    $C_2$, only dies out for $\delta \geq 2p$. So, for $p<\delta<2p$,
    all initial vertices become isolated, but are copied often enough
    such that the number of edges is positive for all times with
    positive probability.
  \end{enumerate}
\end{remark}

\begin{figure}[htb]
  \begin{center}
    \begin{tikzpicture}[xscale=5,yscale=1.6]
      \draw (0.5,2.5) node {(A)};
      \draw [thick,->] (0,0)--(1.1,0);
      \draw [thick,->] (0,-2)--(0,2.1);
      \draw [dashed] (1,-2)--(1,2.1);
      \draw (1.1,-.14) node {$p$};
      \draw (.98,-.14) node {$1$};
%      \draw (.57,-.15) node {$p^\ast$};
%      \draw (.37,-.15) node {$1/e$};
      \draw (0,2.2) node {$\beta_2$};
      
      \draw [thick,domain=0:1] plot ({\x}, {1-2*\x-\x^2});
      \draw (.14,.5) node {\scriptsize$\delta=0$};
      
      \draw [thick,domain=0:1.4146/2] plot ({\x}, {1+ .5-2*\x});
      \draw [thick,domain=1.4146/2:1] plot ({\x}, {1+ 1-2*\x-\x^2});
      \draw (.27,1) node {\scriptsize$\delta=1/2$};
      
      \draw [dashed,domain=0:1] plot ({\x}, {1 - \x- \x*(1-\x)});
      \draw (.3,-.3) node {\scriptsize$\delta=p^2$};
      \draw [->] (.3,-.2)--(.4, .36);

%      \draw [thick,domain=0:.5] plot ({\x}, {1+ .25-2*\x});
%      \draw [thick,domain=.5:1] plot ({\x}, {1+ .5-2*\x-\x^2});
%      \draw (.33,1) node {\scriptsize$\delta=1/2$};
      
      \draw [thick,domain=0:1] plot ({\x}, {1+ 1-2*\x});
      \draw (.43,1.4) node {\scriptsize$\delta=1$};
      
      {\color{white}      \draw (.2,-1) -- (.2,-3);}
    \end{tikzpicture} 
    \begin{tikzpicture}[xscale=5,yscale=1.6]
      \draw (0.5,2.5) node {(B)};
      \draw [thick,->] (0,0)--(1.1,0);
      \draw [thick,->] (0,-3)--(0,2.1);
      \draw [dashed] (1,-3)--(1,2.1);
      \draw (1.1,-.14) node {$p$};
      \draw (.98,-.14) node {$1$};
%      \draw (.57,-.15) node {$p^\ast$};
%      \draw (.37,-.15) node {$1/e$};
      \draw (0,2.2) node {$\beta_3$};
      
      \draw [thick,domain=0:1] plot ({\x}, {1-3*\x-\x^3});
      \draw (.15,.22) node {\scriptsize$\delta=0$};
      
      \draw [thick,domain=0:0.6823] plot ({\x}, {1+ .5-2*\x});
      \draw [thick,domain=0.6823:1] plot ({\x}, {1+ 1.5-3*\x-\x^3});
      \draw (.27,1) node {\scriptsize$\delta=1/2$};
      
      \draw [dashed,domain=0:1] plot ({\x}, {1 - \x- \x*(1-\x^2)/2});
      \draw (.2,-.3) node {\scriptsize$\delta=p \!-\! \frac{p(1-p^2)}{2}$};
      \draw [->] (.3,-.2)--(.4, .432);

%      \draw [thick,domain=0:.5] plot ({\x}, {1+ .25-2*\x});
%      \draw [thick,domain=.5:1] plot ({\x}, {1+ .5-2*\x-\x^2});
%      \draw (.33,1) node {\scriptsize$\delta=1/2$};
      
      \draw [thick,domain=0:1] plot ({\x}, {1+ 1-2*\x});
      \draw (.43,1.4) node {\scriptsize$\delta=1$};

    \end{tikzpicture}
  \end{center}
  \caption{\label{fig2}}Illustration of
  Theorem~\ref{thm:deg-dis-whole-graph}.1. We display the rates of
  decay of $\mathbb E[B_k(t)]$ for $k=2$ (A) and $k=3$ (B).
\end{figure}

\begin{remark}[Connection to previous work]
  We analyzed the case $\delta=0$ previously in \cite{HerPf1}.  We
  note that Theorem~2.7 in that paper is extended by
  Theorem~\ref{thm:deg-dis-whole-graph}, which not only treats the
  case $\delta>0$, but also gives precise exponential decay rates in
  1.\ and 2. Moreover, we add here the almost sure convergence of each
  component of the degree distribution in 3.

  Theorem~2.9 in \cite{HerPf1} is dealing with cliques and $k$-stars
  in the case $\delta=0$ and is extended by
  Theorem~\ref{thm:functionals}. More precisely, since
  $|V_t| \sim e^t$, and \cite{HerPf1} treats the time-discrete model,
  we note that Theorem~2.9(1) of \cite{HerPf1} aligns with
  Theorem~\ref{thm:functionals}.2, but only gives $\cL^1$ (rather than
  $\cL^2$)-convergence. In Theorem 2.9(2) of \cite{HerPf1},
  $S^\circ_k$, the number of $k$-stars in the network at time $t$
  relative to the network size, was analyzed, which coincided with the
  factorial moments of the degree distribution. There, a $k$-star was
  not defined as a sub-graph of $G_t$, since it depended on the order
  of the nodes.  $|V_t|\cdot B_k(t)$ now gives the number of star-like
  sub-graphs in the network at time $t$ consisting of $k+1$
  nodes. Since the only difference between $S^\circ_k$ and $B_k$, as
  given in Theorem~\ref{thm:functionals}.1 is a factor of $k!$, the
  results of \cite{HerPf1} easily apply also for $B_k$ if $\delta=0$.
  Theorem 2.14 of \cite{HerPf1} treats the degrees of initial vertices
  and thus can be compared to Theorem~\ref{thm:functionals}.3.
\end{remark}

\section{Proof of Theorem~1}
\label{S:proof1}
Our analysis of the random graph $PD(p,\delta)$ is based on some main
observations: First, the expected degree distribution can be
represented by a birth-death process with binomial disasters $Z$, such
that the distribution of $Z_t$ equals the expected degree distribution
of $G_t$; see \eqref{eq:lem:duality2}. Second, asymptotics for the
survival probability of such processes were studied in \cite{HerPf2}.

\subsection{Birth-death processes with disasters and $p$-jump
  processes}
\label{ss31}
\begin{definition}\label{def:BDPwDis}
  Let $b>0$, $d\geq0$ and $p\in[0,1]$. Let $Z(b,d,p)=(Z_t)_{t\geq0}$
  be a continuous-time Markov process on $\N_0$ that evolves as
  follows: Given $Z_0=z$, the process jumps
  \begin{itemize}
    \item[--]
      to $z+1$ at rate $bz$;
    \item[--] to $z-1$ at rate $dz$;
    \item[--] to a binomially distributed random variable with
      parameters $z$ and $p$ at rate 1.
  \end{itemize}
  Then we call $Z(b,d,p)$ a \emph{birth-death process subject to binomial disasters}
  with \emph{birth-rate} $b$, \emph{death-rate} $d$ and \emph{survival probability}
  $p$.
\end{definition}

\begin{remark}\label{rem:def:BDPwD} 
  \begin{enumerate}
  \item A birth-death process with binomial disasters, $Z(b,d,p)$
    models the size of a population where each individual duplicates
    with rate $b$ and dies with rate $d$, subjected to binomial
    disasters at rate 1.  These disasters are global events that kill
    off each individual independently of each other with probability
    $1-p$, which generates the binomial distribution in the third part
    of Definition \ref{def:BDPwDis}.
  \item \cite{HerPf2} provides several limit results for such
    branching processes with disasters. As reference for the
    following, let $Z = Z(b,d,p)$ be as above. Then, Corollary~2.7
    of \cite{HerPf2} states:
    \begin{enumerate}
    \item If $b-d\leq p\log\tfrac1p$, $Z$ goes extinct almost surely
      and
      \begin{align*}
        \lim_{t\to\infty}-\tfrac1t\log\Pw(Z_t>0)
        &= (1 - p) - (b-d).
      \end{align*}
    \item If $p\log\tfrac1p<b-d\leq \log\tfrac1p$, $Z$ goes extinct
      almost surely and
      \begin{align*}
        \lim_{t\to\infty}-\tfrac1t\log\Pw(Z_t>0)
        &= 1 - \frac{b-d}{\log\frac1p}\Big(1 + \log\Big(\frac{\log\frac1p}{b-d}\Big)\Big).
      \end{align*}
    \item If $b-d>\log\tfrac1p$, then
      $ \Pw_k(\lim_{t\to\infty}Z_t=0) +
      \Pw_k(\lim_{t\to\infty}Z_t=\infty) = 1 $ and
      \begin{align*}
        \Pw_k(\lim_{t\to\infty}Z_t=\infty)
        &= \Big(1-\frac{d+\log\frac1p}b\Big)\sum_{\ell=1}^k\genfrac(){0pt}{}k\ell(-1)^{\ell-1}
          \prod_{m=1}^{\ell-1}\Big(1-\frac{dm+(1-p^m)}{bm}\Big).
      \end{align*}
    \end{enumerate}
    By constructing a relationship between $PD(p,\delta)$ and
    $Z(p,\delta,p)$ in Lemma \ref{lem:duality}, we are able to
    transfer these results to our duplication graph processes.
  \end{enumerate}
\end{remark}

\begin{lemma}\label{lem:duality}
  Let $p\in(0,1)$, $\delta\geq0$ and recall $F_k(t)$ from Definition
  \ref{def:PDwEdgeDel}. As $h\to0$, the entries $F_k$ of the degree
  distribution yield
  \begin{align}
    \tfrac1{h} & \E[F_k(t+h)-F_k(t)\mid G_t]\notag\\
               &= -(1+pk+\delta k)F_k(t) + p(k-1)F_{k-1}(t) + \delta(k+1)F_{k+1}(t)
                 + \sum_{\ell\geq k}\genfrac(){0pt}{}\ell kp^k(1-p)^{\ell-k}F_\ell(t)+o(1).\label{eq:lem:duality}
  \end{align}
  Moreover, recall $Z := Z(p,\delta,p)$ from
  Definition~\ref{def:BDPwDis} (i.e.\ the binomial distribution of the
  disasters has the birth rate as a parameter) and let
  $\Pw(Z_0=k)=F_k(0)$ for all $k$ be its initial distribution. Then,
  for all $t\geq 0$ and $k$, it holds
  \begin{align}
    \Pw(Z_t=k)=\E[F_k(t)],\label{eq:lem:duality2} 
  \end{align}
  i.e.\ the distribution of $Z_t$ equals the expected degree
  distribution of $G_t$.
\end{lemma}

\begin{proof}
  Letting $\Phi_k(t):=|V_t|F_k(t)$, the absolute number of nodes with degree $k$
  at time $t$, we obtain for $h\to0$ that
  \begin{align*}
    \tfrac1{h}\E[\Phi_k(t+h)-\Phi_k(t)\mid G_t]
    &= -(pk\kappa_t+\delta k)\Phi_k(t) + p(k-1)\kappa_t\Phi_{k-1}(t) + \delta(k+1)\Phi_{k+1}(t)\\[1em]
    &\qquad
      + \sum_{\ell\geq k}\kappa_t\Phi_\ell(t)\genfrac(){0pt}{}\ell kp^k(1-p)^{\ell-k} + O(h),
  \end{align*}
  where the first term on the right hand side stands for the events at
  which a node can lose the degree $k$ by either obtaining a new
  neighbor (by one of its $k$ neighbors being copied, which happens
  with rate $k\kappa_t$, retaining at least the one relevant edge,
  which has probability $p$) or one of its $k$ edges being deleted,
  which happens at rate $\delta k$. The second and third terms
  describe the corresponding gain of a node with degree $k$ by
  analogous events. Finally, the sum equals the rate of a new node
  arising with degree $k$, which can only happen if a node of degree
  $\ell\geq k$ is copied (with rate $\kappa_t\Phi_\ell(t)$) and the
  copy retains exactly $k$ edges
  (which then has a binomial probability).\\
  Now, since $|V_t|$ only increases if a new node is added, i.e. on an
  event related to $\kappa_t$, it follows
  \begin{align*}
    \tfrac1{h}\E[ & F_k(t+h)-F_k(t)\mid F(t)]\\[1em]
     &= \tfrac1{h}\E\Big[\frac{\Phi_k(t+h)}{|V_{t+h}|} -\frac{\Phi_k(t)}{|V_{t+h}|}\Big| F(t)\Big]
       + \tfrac1{h}\E\Big[\frac{\Phi_k(t)}{|V_{t+h}|}  - \frac{\Phi_k(t)}{|V_t|}\Big| F(t)\Big]\\[1em]
     &= \frac{\kappa_t}{|V_t|+1}\Big(-pk\Phi_k(t) + p(k-1)\Phi_{k-1}(t)
             + \sum_{\ell\geq k}\Phi_\ell(t)\genfrac(){0pt}{}\ell kp^k(1-p)^{\ell-k}\Big)\\[1em]
     &\qquad + \frac{- \delta k\Phi_k(t) + \delta(k+1)\Phi_{k+1}(t)}{|V_t|}
             + \kappa_t|V_t|\Phi_k(t)\Big(\frac1{|V_t|+1}-\frac1{|V_t|}\Big) + O(h)\\[1em]
     &= -pkF_k(t)+p(k-1)F_{k-1}(t)+\sum_{\ell\geq k}F_\ell(t)\genfrac(){0pt}{}\ell kp^k(1-p)^{\ell-k}\\[1em]
     &\qquad - \delta kF_k(t) + \delta(k+1)F_{k+1}(t) - F_k(t) + O(h),
  \end{align*}
  and \eqref{eq:lem:duality} holds. Computing the Kolmogorov forwards
  equation for $Z$ shows that for $k=0,1,2,...$
  \begin{align*}
    \tfrac d{dt}\Pw(Z_t=k)
    &= -(1+pk+\delta k)\Pw(Z_t=k) + p(k-1)\Pw(Z_t=k-1) + \delta(k+1)\Pw(Z_t=k+1)\\
    &\qquad + \sum_{\ell\geq k}\Pw(Z_t=\ell)\genfrac(){0pt}{}\ell kp^k(1-p)^{\ell-k},
  \end{align*}
  which is the same relation as \eqref{eq:lem:duality} after taking
  expectation and letting $h\to0$. This shows \eqref{eq:lem:duality2}.
\end{proof}

\subsection{Properties of the piecewise deterministic jump process
  $X$}
\label{ss32}
We have seen the connection of $PD(p,\delta)$ to a branching process
with disasters in Lemma~\ref{lem:duality}. Such branching processes
are in turn closely connected to piecewise deterministic jump
processes as in Lemma~\ref{lem:dual-p-jump-process}
\citep{HerPf2}. Hence, we can now prove
Lemma~\ref{lem:dual-p-jump-process}.

\begin{proof}[Proof of Lemma~\ref{lem:dual-p-jump-process}]
  Lemma \ref{lem:duality} implies that
  $\E[H_x(t)]=\E[(1-x)^{Z_t}]$. Recognizing $(Z_t)=Z(p,\delta,p)$ as a
  homogeneous branching process with disasters $Z^h_{\lambda,q,1,p}$
  in the sense of \citealp[Definition 2.5]{HerPf2}, with death-rate
  $\lambda=p+\delta$ and offspring distribution $q=(q_0,0,q_2,0,\ldots)$
  holding $ q_2 = \frac p{p+\delta} = 1 - q_0 $, the result follows from
  Lemma~4.1 in \cite{HerPf2}. 
\end{proof}

\noindent
For the process $X$, we now obtain a property which is needed in the
proofs of Theorem~\ref{thm:deg-dis-whole-graph} and
Proposition~\ref{prop:limitOfI_x}.

\begin{lemma}[Moments of $X$]\label{lem:momentsOfX}
  Let $X$ be as in Lemma~\ref{lem:dual-p-jump-process}.  If
  $\delta > p-p\log\frac1p$, then $\E_x[X_t^k]=o(\E_x[X_t])$ for all
  $k=2,3,...$ and $\E_x[X_t]\sim ce^{-t(1+\delta-2p)}$, where
  \[
    c :%= c(x_0,p,\delta)
    =x\cdot
    \exp\Big(-p\int_0^\infty\frac{\E_x[X_s^2]}{\E_x[X_s]}ds\Big)\in(0,1).
  \]
\end{lemma}

\begin{proof}
  Recall $\gamma :=\log\frac1p/(p-\delta)$. Indeed, for
  $p-p\log\frac1p<\delta \leq p-p^2\log\frac1p$, such that
  $\gamma\in(p^{-1},p^{-2})$, it follows from Corollary 2.4 of
  \cite{HerPf2} that -- independent of $x$ --
  \begin{align*}
    -\frac1t\log\Big(\frac{\E_x[X_t^2]}{\E_x[X_t]}\Big)
    \xrightarrow{t\to\infty} 1-\frac{1+\log\gamma}\gamma - (1+\delta-2p)
    = 2p\cdot c_2(p,\delta)
    > 0.
  \end{align*}
  On the other hand, if $\delta \geq p-p^2\log\frac1p$, the same
  corollary gives
  \begin{align*}
    -\frac1t\log\Big(\frac{\E_x[X_t^2]}{\E_x[X_t]}\Big)
    \xrightarrow{t\to\infty} &\ 1+2\delta-2p-p^2 - (1+\delta-2p)\\[1em]
    =&\ \delta-p^2
       \geq p^2(\tfrac1p - 1 - \log\tfrac1p)
       > 0.
  \end{align*}
  In either case, there is an $\varepsilon>0$ such that
  $0<r(s):=\E_x[X_s^2]/\E_x[X_s]=O(e^{-\varepsilon s})$
  and it follows from \eqref{eq:log-deriv-moments-of-X}
  \begin{align*}
    \E_x[X_t]
    &= x\exp\Big(-t(1+\delta-2p) - p\int_0^tr(s)ds\Big),
  \end{align*}
  conluding the proof.
\end{proof}

\subsection{Proof of Theorem \ref{thm:deg-dis-whole-graph}}
By \eqref{eq:lem:duality} in Lemma~\ref{lem:duality}, we get that, as $h\to0$
\begin{align*}
  \tfrac{1}{h} \mathbb E[F_0(t+h) - F_0(t)|G_t] &\to -F_0(t) + \delta F_1(t) + \sum_{\ell=0}^\infty(1-p)^\ell F_\ell(t)
                                                = \delta F_1(t)+\sum_{\ell=1}^\infty (1-p)^\ell F_\ell(t)
                                                \geq0.
\end{align*}
Hence, $(F_0(t))_t$ is a bounded sub-martingale and converges almost
surely and in $\cL^1$.  Consequently, the left hand side has to
converge to 0 almost surely. Since $(1-p)^\ell$ is always positive,
that can only be the case if $F_\ell(t)\to0$ almost surely for all
$\ell=1,2,...$, which guarantees almost sure convergence of $F(t)$ to a
vector of the form $F(\infty)=(F_0,0,0,\ldots)$ in all cases.

Let $Z := (Z_t)_{t\geq 0} := Z(p,\delta,p)$ be as in
Definition~\ref{def:BDPwDis}. We note that
$\mathbb E[F_+(t)] = \mathbb P(Z_t>0)$ by \eqref{eq:lem:duality2}.
For 1., we see from Lemma~\ref{lem:dual-p-jump-process} and
Lemma~\ref{lem:momentsOfX} that
\begin{align*}
  \E[F_+(t)]
  &= 1-\E[H_1(t)]
    = 1-\sum_{k=0}^\infty F_k(0)\E_1[(1-X_t)^k]\\
  &= \sum_{k=1}^\infty kF_k(0)\E_1[X_t] + o(\E_1[X_t])
    \sim B_1(0)\cdot ce^{-t(1+\delta-2p)}
\end{align*}
with $c$ as in Theorem~\ref{thm:deg-dis-whole-graph}.1. Moreover, 2.\
follows directly from Corollary~2.7 in \cite{HerPf2}; see
Remark~\ref{rem:def:BDPwD}.2.\ by setting $b=p$ and $d=\delta$. For
3., we again use Corollary~2.7 in \cite{HerPf2}, but use in addition
that
$$ \mathbb E[F_0(t)] = \mathbb P(Z_t=0) = \sum_{k=0}^\infty \mathbb P_k(Z_t=0) \cdot \mathbb P(Z_0=k),$$
and $\sum_{k\geq \ell} F_k(0) \binom k\ell = B_\ell(0)$. \qed

\subsection{Proof of Proposition \ref{prop:limitOfI_x}}
We set $$I_x(t) := 1-H_x(t) = 1 - \sum_{k=0}^\infty (1-x)^k F_k(t).$$
Using the duality relation in \eqref{eq:duality} and Bernoulli's
formula we compute
\begin{align*}
  \E[I_x(t)]
  &= 1 - \sum_{k=0}^\infty F_k(0)\E_x[(1-X_t)^k]\\
  &= 1 - \sum_{\ell=0}^\infty \E_x[X_t^\ell](-1)^\ell\sum_{k\geq\ell}F_k(0)\genfrac(){0pt}{}\ell k\\
  &= 1 - \sum_{\ell=0}^\infty(-1)^\ell\E_x[X_t^\ell]\cdot B_\ell(0)\\
  &= \sum_{\ell=1}^\infty B_\ell(0)(-1)^{\ell+1}\E_x[X_t^\ell].
\end{align*}
For 1., we obtain $\E[I_x(t)]\sim B_1(0)\E_x[X_t]$ from the last
display together with
Lemma~\ref{lem:momentsOfX}.\\
For {3.} suppose that the limits do exist. Then, we see from
Remark~\ref{rem:comvMom} that
$\E_x[X_t^{k+1}]/\E_x[X_t^k]\xrightarrow{t\to\infty}c_k(p,\delta)$ with
$c_k(p,\delta)$ as in \eqref{eq:ces}.  Hence,
\begin{align*}
  \frac{\E[I_x(t)]}{\E_x[X_t]}
  &= \sum_{\ell=1}^\infty B_\ell(0)(-1)^{\ell+1}\prod_{k=1}^{\ell-1}\frac{\E_x[X_t^{k+1}]}{\E_x[X_t^k]}
    \xrightarrow{t\to\infty}
    \sum_{\ell=1}^\infty B_\ell(0)(-1)^{\ell+1}\prod_{k=1}^{\ell-1}c_k(p,\delta).
\end{align*}
Also, the last part of Remark \ref{rem:comvMom} shows, given that
$\E_x[X_t^{k+1}]/\E_x[X_t^k]\xrightarrow{t\to\infty}c_k(p,\delta)$,
that $\E[X_t^k]=o(\E[X_t])$ for all $k\geq2$ such that 2. follows
analogously to 1.\\
Noting that in any case the limit of $\E[I_x(t)]/\E_x[X_t]$ does not
depend on $x$ and using that $I_1(t)=F_+(t)$, we see that
$\E[I_x(t)]/\E[F_+(t)]\sim\E_x[X_t]/\E_1[X_t]$ and finally,
\eqref{eq:P1} is a consequence of \eqref{eq:log-deriv-moments-of-X}.

\section{Proof of Theorem~2}
\label{S:proof2}
The proof of Theorem~\ref{thm:functionals}, which is carried out in
Section~\ref{ss:24}, will be based on the analysis of several
martingales, which are derived in Proposition~\ref{pro:martingales} in
Section~\ref{ss:23}. In Section~\ref{ss:22}, we will analyze the total
size of $G_t$.

\subsection{Two auxiliary functions}\label{ss:21}
We will need two specific functions in the sequel, which we now
analyze.

\begin{lemma}\label{lem:auxiliary-function}\ 
  Let $p\in(0,1)$, $\delta\geq0$ and
  $$g: \begin{cases} [0,\infty) & \to\R, \\ x & \mapsto 1 + \delta x - px - p^x.\end{cases}$$
  Then, $g$ is strictly concave and thus, $x\mapsto g(x)/x$ strictly
  decreases. Also, the following holds:
  \begin{enumerate}
  \item If $\delta\geq p$, $g$ is strictly increasing and,
    $$ g(x) \xrightarrow{x\to \infty} \begin{cases} \infty, & \text{ if }
      p<\delta, \\ 1, & \text{ if }p=\delta.\end{cases}$$
  \item If $p-\log\frac1p < \delta < p$,
    $$ \text{$g$ is }\begin{cases} \text{strictly increasing on
        $(0,\xi)$,} \\ \text{strictly decreasing on
        $(\xi,\infty)$}\end{cases}$$ for $\xi:=\log\gamma/\log\frac1p$
    with $\gamma:=\log\frac1p/(p-\delta)$. The global maximum is
    $g(\xi) = 1 - \frac1\gamma(1+\log\gamma)$.
  \item If $\delta\leq p-\log\frac1p$,
    $g$ strictly decreases and its maximum is $g(0)=0$.
  \end{enumerate}
\end{lemma}

\begin{proof}
  All results are straight-forward to compute. First,
  $g'(x)=\delta-p+\log\frac1p\cdot p^x$ for all cases. Since the right
  hand side strictly increases, $g$ is strictly concave. 1.\ follows
  from the form of $g'$. For 3., we have that
  $g'(x) \leq \delta - p + \log \frac 1p\leq 0$, implying the
  result. For 2., we have that $g'(x) = 0$ iff
  $p^{-x} = \log\frac 1p /(p-\delta)=\gamma$ iff
  $x = \log\gamma/\log\frac 1p = \xi = \log\gamma/(\gamma(p-\delta))$
  and the rest follows. 
\end{proof}

\begin{lemma}\label{lem:hr}
  Let $g^r(n) := \frac{\Gamma(n+r)}{\Gamma(n)}$ and $n_0\geq 2$. Then,
  there are $0<c_r\leq 1 < C_r < \infty$, such that
  \begin{align}
    \label{eq:cr}
    c_r n^r \leq g^r(n) \leq C_r n^r \text{ for all }n\geq n_0.
  \end{align}
\end{lemma}

\begin{proof}
  First, we note that $g^r(n)\sim n^r$ as $n\to\infty$ (see e.g.\
  6.1.46.\ of \citealp{abramowitz}) and hence, the result follows.
\end{proof}

\subsection{The size of the graph}\label{ss:22}
For the asymptotics of the functionals of the random graph in Theorem
\ref{thm:functionals} it will be helpful to understand the asymptotics
of the process $(|V_t|)$.  Here and below, we will frequently use the
following well-known lemma.

\begin{lemma}\label{l:martGen}
  Let $X = (X_t)_{t\geq 0}$ be a Markov process with complete and
  separable state space $(E,r)$, and $f: E \to\mathbb R$ continuous
  and bounded and such that
  $$\lim_{h\to 0} \tfrac 1h \E[f(X_{t+h}) - f(X_t)|X_t=x] = \lambda f(x), \qquad x\in E$$
  for some $\lambda \in\mathbb R$, then
  $(e^{-t\lambda } f(X_t))_{t\geq 0}$ is a martingale.
\end{lemma}

\begin{proof}
  See Lemma 4.3.2 of \cite{EK86}.
\end{proof}

\begin{lemma}[Graph size]\label{lem:yule-process}
  Let $g^r(n) := \Gamma(n+r)/\Gamma(n)$. For all $r>-(|V_0|+1)$, the
  process $(e^{-tr} g^r(|V_t|+1))_{t\geq 0}$ is a non-negative
  martingale. Moreover, there is a random variable $V_\infty$ such
  that the following holds:
  \begin{align*}
    &e^{-t}|V_t| \xrightarrow{t\to\infty}V_\infty \text{ almost surely and in $\mathcal L^r$ for all $r\geq 1$},
    \\
    & e^{t}/|V_t| \xrightarrow{t\to\infty}1/V_\infty \text{ almost surely and in $\mathcal L^r$ for $1\leq r < |V_0|+1$},
    \\
    & V_\infty \text{ is $\Gamma(|V_0|+1,1)$-distributed.}
  \end{align*}
\end{lemma}

\begin{proof}
  Let $0<c_r < 1 < C_r<\infty$ be as in Lemma~\ref{lem:hr}.  The
  process $V :=(|V_t|)_{t\geq 0}$ is a Markov process which jumps from
  $v$ to $v+1$ at rate $v+1$. Setting $g^r(v)=\Gamma(v+r)/\Gamma(v)$,
  we see that the process $(g^r(|V_t|+1))_{t\geq0}$ is well-defined
  and non-negative if $|V_t|+1+r>0$ for all $t$, i.e. if
  $r>-(|V_0|+1)$. Then, as $h\to 0$,
  \begin{align*}
    \tfrac 1h \E[g^r(V_{t+h}) - g^r(V_t)|V_t=v] & = (v+1) (g^r(v+2) - g^r(v+1)) + o(1)
    \\ & = (v+1)g^r(v+1) \Big( \frac{v+1+r}{v+1}-1\Big) + o(1)
         = rg^r(v+1) + o(1)
  \end{align*}
  and Lemma~\ref{l:martGen} implies that
  $(e^{-tr}g^r(|V_t|+1))_{t\geq0}$ is a (non-negative) martingale for
  all $r>-(|V_0|+1)$, and therefore $\cL^1$-bounded.  By the
  martingale convergence theorem, this martingale converges almost
  surely. Using \eqref{eq:cr}, the martingale
  $(e^{-t}g^1(|V_t|+1))_t=(e^{-t}(|V_t|+1))_t$ is $\cL^r$-bounded for
  every $r\geq1$ and therefore converges in $\cL^r$. Analogously, for
  $r=-1$, the martingale $(e^{t}g^{-1}(|V_t|+1))_t=(e^{t}/|V_t|)_t$
  is $\cL^r$-bounded for $1\leq r<|V_0|+1$ and converges in $\cL^r$.\\[.5em]
  Noting that $(|V_t|+1)_{t\geq 0}$ is a Yule-process starting in
  $|V_0|+1$, we have that $|V_t|+1$ is distributed as the sum of
  $|V_0|+1$ independent, geometrically distributed random variables
  with success probabilities $e^{-t}$ (see e.g.\ p.\ 109 of
  \citealp{AthreyaNey1972}). Hence, as $t\to\infty$, we find that
  $e^{-t}|V_t|$ converges in distribution to the sum of $|V_0|+1$
  independent, exponentially distributed random variables with unit
  rate. This is a $\Gamma(|V_0|+1,1)$ distribution.
\end{proof}

\subsection{Some martingales}\label{ss:23}
\noindent
Similarly to the discrete-time pure duplication graph in \cite{HerPf1}
we obtain martingales for the functionals of $PD(p,\delta)$.

\begin{proposition}[Martingales]\label{pro:martingales}\
  \begin{enumerate}
  \item Considering the function $g$ of Lemma
    \ref{lem:auxiliary-function}, it holds
    \begin{enumerate}
    \item
      if $g(1)\leq g(k)$, $(e^{tg(1)}B_k(t))_{t\geq0}$ is a martingale
      that almost surely converges to a limit $B_k(\infty)\in\cL^1$.
    \item if $g(1)>g(k)$, there is a process $R_k(t)$ such that
      $(e^{tg(k)}(B_k(t) + R_k(t)))_{t\geq 0}$ is a positive
      martingale that almost surely converges to a limit
      $B_k(\infty)\in\cL^1$ and
      $e^{tg(k)}R_k(t) \xrightarrow{t\to\infty} 0$. In particular,
      $e^{tg(k)}B_k(t)\to B_k(\infty)$ almost surely as $t\to\infty$.
    \end{enumerate}
    Combining (a) and (b), we find
    $e^{t(g(1) \wedge g(k))}B_k(t) \xrightarrow{t\to\infty}
    B_k(\infty)\in \cL^1$.
  \item For $k=2,3,...$,
    $(e^{-t(kp^{k-1}-1-\delta\genfrac(){0pt}{}k2)}C_k(t)/|V_t|)_{t\geq0}$
    is a martingale that converges almost surely to a limit
    $\tilde C_k(\infty)$.  If additionally $C_k(0)>0$ and
    $\delta<2p^{k-1}/(k-1)$, the convergence also holds in $\cL^2$.
  \item Let $i\leq|V_0|$. Then,
    $(e^{-t(p-\delta-1)}D_i(t)/|V_t|)_{t\geq0}$ is a martingale that
    converges almost surely to a limit $\tilde
    D_i(\infty)\in\cL^1$. Moreover, 
    \begin{align}
      \label{eq:Di1}
      \E[e^{-t(p-\delta)} D_i(t)] & = D_i(0)\Big(1 + (1-e^{-t})\frac{p}{|V_0|}\Big)
    \end{align}
    and for $r\geq 2$, there is $C>0$, depending only on $r,p,\delta$
    such that
    \begin{align}
      \label{eq:Di2}
      \mathbb E[(e^{-t(p-\delta)}D_i(t))^r] \leq 
      \E[(D_i(0))^r] + C\int_0^t e^{-s(p-\delta)} \mathbb
      E[(e^{-s(p-\delta)}D_i(s))^{r-1}]ds.
    \end{align}    
  \end{enumerate}
\end{proposition}

\begin{proof}
  1. Since the sum in $B_k(t)$ is almost surely finite for every $k$
  and $t$, it follows for $h\to0$ using equation
  \eqref{eq:lem:duality}, that
  \begin{align*}
    \tfrac1{h}\E[B_k & (t+h)-B_k(t)|G_t]\\[1em]
                     &= \sum_{\ell\geq k}\genfrac(){0pt}{}\ell k\Bigg(-(1+p\ell+\delta\ell)F_\ell(t)
                       + p(\ell-1)F_{\ell-1}(t) + \delta(\ell+1)F_{\ell+1}(t)\\[.5em]
                     &\qquad\qquad\qquad
                       + \sum_{m\geq\ell}\genfrac(){0pt}{}m\ell p^\ell(1-p)^{m-\ell}F_m(t)\Bigg) + o(1)\\[1em]
                     &= -B_k(t) + p(k-1)F_{k-1}(t)
                       + \sum_{m\geq k}F_m(t)\sum_{\ell=k}^m
                       \genfrac(){0pt}{}\ell k\genfrac(){0pt}{}m\ell p^\ell(1-p)^{m-\ell}\\[.5em]
                     &\qquad\quad + \sum_{\ell\geq k}F_\ell(t)\Bigg(\underbrace{
                       - (p+\delta)\ell\genfrac(){0pt}{}\ell k
                       + p\ell\genfrac(){0pt}{}{\ell+1}k
                       + \delta\ell\genfrac(){0pt}{}{\ell-1}k
                       }_{=:a(\ell,k)}\Bigg)+o(1).
  \end{align*}
  Considering that
  $\genfrac(){0pt}{}{n+1}m-\genfrac(){0pt}{}nm =
  \genfrac(){0pt}{}n{m-1}$ and
  $\frac nm \cdot \genfrac(){0pt}{}{n-1}{m-1} = \genfrac(){0pt}{}nm$,
  we deduce
  \begin{align*}
    a(\ell,k)
    &= p\ell\genfrac(){0pt}{}\ell{k-1} - \delta\ell\genfrac(){0pt}{}{\ell-1}{k-1}
      = (p-\delta)\ell\genfrac(){0pt}{}{\ell-1}{k-1} + p\ell\genfrac(){0pt}{}{\ell-1}{k-2}\\[1em]
    &= (p-\delta)k\genfrac(){0pt}{}\ell k + p(k-1)\genfrac(){0pt}{}\ell{k-1},
  \end{align*}
  which implies that
  \begin{align*}
    \tfrac1{h} & \E[B_k(t+h)-B_k(t)|G_t]\\[.5em]
    =& -B_k(t) + p(k-1)B_{k-1}(t) + (p-\delta)kB_k(t)
    \\ & \qquad \qquad \qquad + \sum_{m\geq k}F_m(t)
         \underbrace{\sum_{\ell=0}^{m-k}
         \genfrac(){0pt}{}{m-k}\ell\genfrac(){0pt}{}mk p^{\ell+k}(1-p)^{m-k-\ell}
         }_{=\genfrac(){0pt}{}mkp^k} + o(1)\\
    =& B_k(t)\big((p-\delta)k - (1-p^k)\big) + p(k-1)B_{k-1}(t) + o(1)\\[1em]
    =& -g(k)B_k(t) + p(k-1)B_{k-1}(t)+ o(1),
  \end{align*}
  recalling the function $g:x\mapsto 1+\delta x-px-p^x$ from Lemma
  \ref{lem:auxiliary-function}. In any case we see that
  $(e^{tg(1)}B_1(t))_{t\geq0}$ is a non-negative martingale converging
  almost surely to a limit $B_1(\infty)\in\cL^1$. For $k=2,3,...$ let
  $g_{\min}(k):=\min_{1\leq\ell\leq k}g(k)$
  the running minimum of $g$. Then, there are two cases to consider:\\[.5em]
  1. $g(1)\leq g(k)$: It holds by strict concavity of $g$ (see
  Lemma~\ref{lem:auxiliary-function}) that in this case
  $g(1)=g_{\min}(k)<g(\ell)$ for all $1<\ell<k$.  Thus, letting
  \[
    \lambda_m^k
     := \frac{g(1)}{g(k)}\prod_{\ell=m}^{k-1}\frac{p\ell}{g(\ell)-g(1)}, \qquad m=2,...,k
  \]
  and $\lambda_1^k:=1+\frac{1}{g(1)}p\lambda_2^k$, these coefficients are well-defined
  and positive. Considering the linear combination $Q_k(t):=\sum_{m=1}^k\lambda_m^kB_m(t)$
  we obtain
  \begin{align*}
    \tfrac1{h} & \E[Q_k(t+h)-Q_k(t)|G_t]\\[.5em]
               &= -g(1)B_k(t) + \sum_{m=1}^{k-1}B_m(t)(-\lambda_m^kg(m) + \lambda_{m+1}^kpm) + o(1)\\[1em]
               &= -g(1)B_k(t) + \sum_{m=2}^{k-1}B_m(t)\lambda_m^k\Big(-g(m) + pm\frac{g(m)-g(1)}{pm}\Big)
                 + B_1(t)\Big(-g(1)\lambda_1^k + p\lambda_2^k\Big) + o(1)\\[.5em]
               &= -g(1)Q_k(t) + o(1).
  \end{align*}
  So now, $(e^{tg(1)}Q_k(t))_{t\geq0}$ is a non-negative martingale
  for every $k$. Since $B_k(t)$ can be represented as a linear
  combination of $(Q_\ell(t))_{1\leq\ell\leq k}$, also
  $(e^{tg(1)}B_k(t))_{t\geq0}$ has to be a non-negative martingale and
  thus converges
  to some $B_k(\infty)\in\cL^1$.\\[.5em]
  2. $g(1)>g(k)$: Here it holds by strict concavity of $g$ that
  $g(\ell)>g(k)=g_{\min}(k)$ for all $\ell=1,...,k-1$. Hence
  \[
    \lambda_m^k := \prod_{\ell=m}^{k-1}\frac{p\ell}{g(\ell)-g(k)},
    \qquad m=1,...,k,
%      = \prod_{\ell=m}^{k-1}\frac{p\ell}{(p-\delta)(k-\ell)+p^k-p^\ell},
  \]
  are well-defined and positive. We compute analogously to the first
  case that, as $h\to 0$, with
  $Q_k(t) := \sum_{m=1}^k \lambda_m^k B_m$,
  \begin{align*}
    \tfrac1{h} & \E[Q_k(t+h)-Q_k(t)|G_t]\\[.5em]
%     &= -g(k)B_k(t) + \sum_{m=1}^{k-1}B_m(t)(-\lambda_m^kg(m) + \lambda_{m+1}^kpm)\\[1em]
               &= -g(k)B_k(t) + \sum_{m=1}^{k-1}B_m(t)\lambda_m^k\Big(-g(m) + pm\frac{g(m)-g(k)}{pm}\Big) + o(1)
                 = -g(k)Q_k(t) + o(1).
  \end{align*}
  Thus, $(e^{tg(k)}Q_k(t))_{t\geq0}$ is a non-negative martingale and
  has an almost sure limit $B_k(\infty)\in\cL^1$. Moreover, for
  $\ell<k$, since $g(\ell) > g(k)$, we have that
  $e^{tg(k)} Q_\ell(t) \to 0$, so, writing $R_k(t) = Q_k(t) - B_k(t)$,
  we see that $R_k(t) = \sum_{\ell=1}^{k-1} \mu_\ell^k Q_\ell(t)$ for some
  $\mu_1^k,...,\mu_{k-1}^k$ and $e^{tg(k)}R_k(t)\to 0$ and
  $e^{tg(k)}B_k(t)\to B_k(\infty)$ follows.\\[.5em]
  2. For the cliques fix $t\geq0$ and let $N_k(v)$ for every node
  $v\in V_t$ denote the number of $k$-cliques that node is part
  of. Then, $\sum_{v\in V_t}N_k(v)=kC_k(t)$. Analogously define
  $M_k(e)$ as the number of cliques that the edge $e\in E_t$ is
  contained in, such that
  $\sum_{e\in E_t}M_k(e)=\genfrac(){0pt}{}k2C_k(t)$.  Also, let
  $\tilde C_k(t):=C_k(t)/|V_t|$. Note that for a new $k$-clique to
  arise, a node $v$ inside of such a clique has to be copied. Then,
  every of the $N_k(v)$ cliques $v$ is part of has a chance of
  $p^{k-1}$ that the copy obtains the $k-1$ edges it needs to form a
  new $k$-clique. Also, whenever an edge $e$ is deleted, all $M_k(e)$
  $k$-cliques are destroyed.  We deduce
  \begin{align}
    \tfrac1h\E[\tilde C_k(t+h) - \tilde C_k(t)|G_t]
     &= \frac{|V_t|+1}{|V_t|}\sum_{v\in V_t}\Big(\frac{C_k(t)+p^{k-1}N_k(v)}{|V_t|+1}-\tilde C_k(t)\Big)
       - \delta\sum_{e\in E_t}\frac{M_k(e)}{|V_t|} + o(1)\notag\\[1em]
     &= \sum_{v\in V_t}\Big(\tilde C_k(t) + \frac{p^{k-1}N_k(v)}{|V_t|} - \frac{|V_t|+1}{|V_t|}\tilde C_k(t)\Big)
       - \frac{\delta}{|V_t|}\genfrac(){0pt}{}k2C_k(t) + o(1)\notag\\[1em]
     &= \tilde C_k(t)\Big(\underbrace{kp^{k-1} - 1 - \delta\genfrac(){0pt}{}k2}_{=:q_k}\Big) + o(1).\label{eq:cliques}
  \end{align}
  This shows that $(e^{-tq_k}\cdot\tilde C_k(t))_{t\geq0}$ is a
  non-negative martingale and
  hence converges almost surely to an integrable random variable $\tilde C_k(\infty)$.\\[.5em]
  It remains to show the $\cL^2$-convergence of the martingale
  $(e^{-tq_k}\tilde C_k(t))_{t\geq0}$ for $q_k+1>0$,
  i.e. $\delta<2p^{k-1}/(k-1)$. This will be done by considering the
  number of (unordered) pairs of $k$-cliques,
  $CC_k(t):=\genfrac(){0pt}{}{C_k(t)}2=(C_k(t)^2-C_k(t))/2$, and
  verifying that the process given by
  $\widetilde{CC}_k(t):=e^{-t\cdot2q_k}\frac{CC_k(t)}{|V_t|(|V_t|-1)}$
  is $\cL^1$-bounded, which implies $\cL^2$-boundedness of the
  martingale $(e^{-tq_k}\cdot\tilde C_k(t))_{t\geq0}$
  and concludes the proof.\\[.5em]
  Let us denote by $C_{k,\ell}(t)$ the number of (unordered) pairs of
  $k$-cliques which have exactly $\ell$ shared vertices. Since the
  \emph{overlap} of such a pair (i.e. the sub-graph both cliques have
  in common) is an $\ell$-clique with $\binom\ell2$ edges, the number
  of edges making up the pair equals $2\binom k2- \binom \ell2$.
  Hence, arguing as in the proof of Theorem~2.9 in \cite{HerPf1},
  considering that (i) one new such pair arises if one of the
  $2(k-\ell)$ non-shared vertices is fully copied (probability
  $p^{k-1}$), and (ii) one new pair arises if one of the $\ell$ shared
  vertices is fully copied (probability $p^{2k-\ell-1}$), by taking
  the copied node instead of the original one; in addition, there are
  two new pairs of $k$-cliques, one original and one copied, which
  share $\ell-1$ vertices, and (iii) if one of the $\ell$ shared
  vertices is chosen, but only one of the two cliques is fully copied
  (probability $2p^{k-1}(1-p^{k-\ell}$)) one new pair of $k$-cliques
  arises, which shares $\ell-1$ vertices. In addition, such a pair
  will be destroyed if one of its edges is deleted, hence we deduce
  for $\ell\leq k-2$
  \begin{align*}
    \tfrac1{h}\E[ & C_{k,\ell}(t+h)-C_{k,\ell}(t)\mid G_t]\\[1em]
                  &= (|V_t|+1)\cdot\Big(\frac{2(k-\ell)p^{k-1} + \ell p^{2k-\ell-1}}{|V_t|}C_{k,\ell}(t)
                    + \frac{2(\ell+1)p^{k-1}}{|V_t|}C_{k,\ell+1}(t)\Big)\\[1em]
                  &\qquad\qquad\qquad\qquad\qquad\qquad\qquad\qquad\qquad\qquad\qquad
                    - \delta\cdot\Big(2\genfrac(){0pt}{}k2-\genfrac(){0pt}{}\ell2\Big)\cdot C_{k,\ell}(t) + o(1),
  \end{align*}
  which implies for
  $\widetilde
  C_{k,\ell}(t):=e^{-t\cdot2q_k}\frac{C_{k,\ell}(t)}{|V_t|(|V_t|-1)}$,
  that
  \begin{align*}
    \tfrac1{h}\E[ & \widetilde C_{k,\ell}(t+h)-\widetilde C_{k,\ell}(t)\mid G_t]\\
                  &= -2q_k\widetilde C_{k,\ell}(t) + e^{-t\cdot2q_k}(|V_t|+1)\\
                  &\qquad\qquad\quad\cdot\Bigg( \frac{(2(k-\ell)p^{k-1} + \ell
                    p^{2k-\ell-1})C_{k,\ell}(t) + 2(\ell+1)p^{k-1}C_{k,\ell+1}(t)}
                    {|V_t|\cdot(|V_t|+1)|V_t|}\\
                  &\qquad\qquad\qquad\qquad +
                    C_{k,\ell}(t)\cdot\Big(\frac1{(|V_t|+1)|V_t|} -
                    \frac1{|V_t|(|V_t|-1)}\Big)
                    \Bigg)\\
                  &\qquad\qquad\qquad\qquad\qquad\qquad\qquad\qquad -
                    2\delta\genfrac(){0pt}{}k2\widetilde C_{k,\ell}(t)
                    + \delta\genfrac(){0pt}{}\ell2\widetilde C_{k,\ell}(t) + o(1)\\
                  &= \widetilde C_{k,\ell}(t)\Bigg(-2q_k - 2\delta\genfrac(){0pt}{}k2 +
                    \delta\genfrac(){0pt}{}\ell2
                    + (2(k-\ell)p^{k-1} + \ell p^{2k-\ell-1})\cdot\frac{|V_t|-1}{|V_t|} - 2\Bigg)\\
                  &\qquad
                    + \widetilde C_{k,\ell+1}(t)\cdot2(\ell+1)p^{k-1}\cdot\frac{|V_t|-1}{|V_t|}  + o(1)\\
                  & \leq \widetilde C_{k,\ell}(t)\Big( -2\ell p^{k-1} + \ell
                    p^{2k-\ell-1} + \delta\genfrac(){0pt}{}\ell2\Big)
                    + \widetilde C_{k,\ell+1}(t)\cdot2(\ell+1)p^{k-1} + o(1)\\
                  &= -\ell\widetilde C_{k,\ell}(t)\Big( p^{k-1}(2 - p^{k-\ell}) -
                    \tfrac\delta2(\ell-1)\Big) + 2(\ell+1)p^{k-1}\widetilde
                    C_{k,\ell+1}(t) + o(1).  \intertext{Analogously, for $\ell=k-1$,
                    additional pairs arise if a clique with $k$ vertices is completely
                    copied (probability $p^{k-1}$), so}
                    \tfrac1{h}\E[ & \widetilde C_{k,k-1}(t+h)-\widetilde C_{k,k-1}(t)\mid G_t]\\
                  &\leq -(k-1)\widetilde C_{k,k-1}(t)\Big( p^{k-1}(2 - p) -
                    \tfrac\delta2(k-2)\Big) +
                    2kp^{k-1}e^{-t\cdot2q_k}\cdot\frac{C_k(t)}{|V_t|(|V_t|-1)} + o(1).
                    \intertext{Also, letting
                    $\widehat C_k(t):=e^{-t\cdot2q_k}C_k(t)/(|V_t|(|V_t|-1))$ and
                    combining the calculation above with the one in \eqref{eq:cliques},
                    it follows that }
                    \tfrac1{h}\E[ & \widehat C_k(t+h)-\widehat C_k(t)\mid G_t]\\
                    % &= -2q_k\widehat C_k(t) +
                    % e^{-t\cdot2q_k}(|V_t|+1)\Bigg(\frac{kp^{k-1}(C_k(t)+1)}{|V_t|^2(|V_t|+1)}
                    % + \frac{|V_t|-k +
                    % k(1-p^{k-1})}{|V_t|}\cdot\frac{C_k(t)}{|V_t|(|V_t|+1)}\Bigg) %\\[1em]
                    % &\qquad
                    % - \delta\genfrac(){0pt}{}k2\widehat C_k(t)+ o(1)\\[1.5em]
                  &= -2q_k\widehat C_k(t)
                    + e^{-t\cdot2q_k}\Bigg(\frac{kp^{k-1}C_k(t)}{|V_t|(|V_t|-1)}\cdot\frac{|V_t|-1}{|V_t|}
                    + \frac{C_k(t)}{|V_t|} - C_k(t)\frac{|V_t|+1}{|V_t|(|V_t|-1)}\Bigg) %\\[1em]
                    % &\qquad
                        - \delta\genfrac(){0pt}{}k2\widehat C_k(t)+ o(1)\\
                  &\leq \widehat C_k(t)\Big(-2q_k + kp^{k-1} - 2 - \delta\genfrac(){0pt}{}k2\Big) + o(1)\\
                  &= -\Big(kp^{k-1} - \delta\genfrac(){0pt}{}k2\Big)\widehat C_k(t)
                    = -k\widehat C_k(t)\Big(p^{k-1}(2-p^0) - \tfrac\delta2(k-1)\Big) + o(1).
  \end{align*}
  Now, since
  \begin{align*}
    &\qquad
    \delta
      < \frac{2p^{k-1}}{k-1}
      = \min_{2\leq m\leq k}\Big\{\frac{2p^{k-1}(2-p^{k-m})}{m-1}\Big\},
  \intertext{the coefficients given by}
    &\qquad
    \lambda_\ell
     := \prod_{m=1}^\ell\frac{2p^{k-1}}{p^{k-1}(2-p^{k-m})-\frac\delta2(m-1)}
  \intertext{for $1\leq\ell\leq k$ are well-defined and positive and we obtain for the
    linear combination
  }
    &\qquad
    R_k(t)
     := \widetilde C_{k,0}(t)
       + \sum_{\ell=1}^{k-1}\lambda_\ell\widetilde C_{k,\ell}(t)
       + \lambda_k\widehat C_k(t)
  \intertext{that}
    \lim_{h\to0}\tfrac1h\E[ & R_k(t+h)-R_k(t)\mid G_t]\\[1em]
     &\leq \sum_{\ell=1}^{k-1}\widetilde C_{k,\ell}(t)\cdot\ell\Big(
            - \lambda_\ell\big(p^{k-1}(2-p^{k-\ell})-\tfrac\delta2(\ell-1)\big)
            + \lambda_{\ell-1}2p^{k-1}\Big)\\[1em]
     &\qquad
          + \widehat C_k(t)\cdot k\Big(
            - \lambda_k\big(p^{k-1}(2-p^{k-k})-\tfrac\delta2(k-1)\big)
            + \lambda_{k-1}2p^{k-1}\Big)
     = 0.
  \end{align*}
  Thus, $(R_k(t))$ is a non-negative super-martingale, $\cL^1$-bounded
  and, since
  $\lambda_{\min}:=\min(\{1\}\cup\{\lambda_\ell;1\leq\ell\leq k\})>0$
  and
  $\widetilde{CC}_k(t)\leq R_k(t)/\lambda_{\min}$, the proof of 2. is complete.\\[1em]
  3. For the degree $D_i(t)$, we set $\tilde D_i(t)=D_i(t))/|V_t|$ and
  compute, as $h\to 0$,
  \begin{align*}
    \tfrac1{h}\E[ & \tilde D_i(t+h) - \tilde D_i(t)\mid G_t]\\[1em]
                  &= (|V_t|+1)\Bigg(
                    \frac{pD_i(t)}{|V_t|}\Big(\frac{D_i(t)+1}{|V_t|+1}-\frac{D_i(t)}{|V_t|}\Big)
                    + \Big(1-\frac{pD_i(t)}{|V_t|}\Big)\Big(\frac{D_i(t)}{|V_t|+1}-\frac{D_i(t)}{|V_t|}
                    \Big)\Bigg)\\[1em]
                  &\qquad\qquad\qquad\qquad\qquad\qquad\qquad\qquad\qquad\qquad
                    + \delta D_i(t)\Big(\frac{D_i(t)-1}{|V_t|}-\frac{D_i(t)}{|V_t|}\Big) + o(1)\\[1em]
                  &= \frac{D_i(t)}{|V_t|}\Bigg(
                    p  \frac{|V_t| - D_i(t)}{|V_t|} 
                    - \Big(1-\frac{pD_i(t)}{|V_t|}\Big)  -
                    \delta \Bigg) + o(1)
    \\ & = \tilde D_i(t)(p-\delta-1) + o(1).
  \end{align*}
  Lemma~\ref{l:martGen} shows that
  $(e^{-t(p-\delta-1)}D_i(t)/|V_t|)_{t\geq0}$ is a non-negative
  martingale, and hence converges almost surely. Furthermore, we write
  with $g^r(n) := \Gamma(n+r)/\Gamma(n)$
  \begin{equation}
    \label{eq:9213}
    \begin{aligned}
      \tfrac 1h \E[ & g^r(D_i(t+h) - g^r(D_i(t))\mid G_t] \\ & =
      (|V_t|+1) \frac{pD_i(t)}{|V_t|}(g^r(D_i(t)+1) - g^r(D_i(t))) +
      \delta D_i(t) (g^r(D_i(t)-1) - g^r(D_i(t))) + o(1) \\ & =
      g^r(D_i(t)) \Big(pr\frac{|V_t|+1}{|V_t|} + \delta
      D_i(t)\Big(\frac{D_i(t)-1}{D_i(t) + r-1} - 1\Big)\Big) + o(1) \\
      & = g^r(D_i(t)) r(p-\delta) + g^r(D_i(t)) r\Big(
      p\Big(\frac{|V_t|+1}{|V_t|}-1\Big) - \delta
      \Big(\frac{D_i(t)}{D_i(t) + r-1} -1\Big)\Big) + o(1) \\ & =
      g^r(D_i(t)) r(p-\delta) + g^r(D_i(t)) r \Big( p \frac{1}{|V_t|}
      + \delta(r-1) \frac{1}{D_i(t) + r-1}\Big) + o(1).
    \end{aligned}
  \end{equation}
  Letting $h\to 0$, this gives \eqref{eq:Di1} for $r=1$ since, using
  the martingale $(e^{-t(p-\delta-1)}D_i(t)/|V_t|)_{t\geq 0}$,
  \begin{align*}
    \E[e^{-t(p-\delta)}D_i(t)] & = D_i(0) + \int_0^t p \E[e^{-s(p-\delta)} D_i(t)/|V_s|] ds
    \\ & = D_i(0) + \int_0^t p e^{-s} D_i(0)/|V_0| ds.
  \end{align*}
  Moreover, since
  $\frac{g^r(D_i(t))}{D_i(t) + r-1} = g^{r-1}(D_i(t))$,
  \eqref{eq:9213} gives for some $C>0$, depending on $r,p,\delta$
  \begin{align*}
    \frac{d}{dt}\E[ e^{-tr(p-\delta)}g^r(D_i(t))] & \leq  C \cdot \E[e^{-tr(p-\delta)} g^{r-1}(D_i(t))],
  \end{align*}
  and \eqref{eq:Di2} follows with Lemma~\ref{lem:hr} and integration.
\end{proof}

\subsection{Proof of Theorem~\ref{thm:functionals}}
\label{ss:24}
1. Recalling the function $g$ from Lemma~\ref{lem:auxiliary-function},
we note that (see also Remark\ref{lem:remInt}.1 for the second
equality)
$\beta_k = (1+\delta-2p) \wedge 1+(\delta - p)k - p^k = g(1)\wedge
g(k)$. Hence, Lemma \ref{pro:martingales}.1 shows that
$e^{t\beta_k}B_k(t)$ is non-negative and converges to some
$B_k(\infty)\in\cL^1$. So, 1. follows.
\\[.5em]
2. We combine Lemma \ref{pro:martingales}.2 (recall the random
variable $\tilde C_k(\infty)$) with the almost sure convergence
$e^{-t}V_t \xrightarrow{t\to\infty} V_\infty$ from
Lemma~\ref{lem:yule-process}. In all cases, we have that
\begin{equation}
  \label{eq:proof21}
  \begin{aligned}
    \exp\Big(-t\Big(kp^{k-1} - \delta\binom k2\Big)\Big)C_k(t) & =
    \exp\Big(-t\Big(kp^{k-1} - 1- \delta\binom
    k2\Big)\Big)C_k(t)/|V_t| \cdot e^{-t} |V_t| \\ &
    \xrightarrow{t\to\infty} \tilde C_k(\infty)\cdot V_\infty =:
    C_k(\infty),
  \end{aligned}  
\end{equation}
where $V_\infty>0$ almost surely.\\
If $\delta\geq 2p^{k-1}/(k-1)$, it is
$kp^{k-1} - \delta\binom k2 \leq 0$ and the convergence can only hold
if $C_k(t) \xrightarrow{t\to\infty} 0$ almost surely. Since
$C_k(t)\in\N_0$, the first hitting time $T^{C_k}$ of 0 has to be
finite. On the other hand, for $\delta<2p^{k-1}/(k-1)$, combining the
$\cL^2$-convergences in Lemma \ref{pro:martingales}.2 and Lemma
\ref{lem:yule-process} we obtain that the convergence in
\eqref{eq:proof21} also holds in $\cL^1$.  Since
$(e^{-t(kp^{k-1}-1-\delta\genfrac(){0pt}{}k2)}C_k(t)/|V_t|)$ is an
$\cL^2$-convergent and thus uniformly integrable martingale,
$\Pw(C_k(\infty)>0)
=\Pw(\tilde C_k(\infty)>0)>0$.\\[.5em]
3. Finally, fix $i\in\{1,\ldots,|V_0|\}$. Again, we combine
Lemma~\ref{pro:martingales}.3 (recall the random variable
$\tilde D_i(\infty)$) with the almost sure convergence
$e^{-t}V_t \xrightarrow{t\to\infty} V_\infty$ from
Lemma~\ref{lem:yule-process}. In all cases, we have that
\begin{equation}
  \label{eq:proof31}
  \begin{aligned}
    e^{-t(p-\delta)} D_i(t) = e^{-t(p-\delta-1)}\frac{D_i(t)}{|V_t|}
    e^{-t} |V_t| \xrightarrow{t\to\infty} \tilde D_i(\infty)\cdot
    V_\infty =: D_i(\infty).
  \end{aligned}  
\end{equation}
If $\delta<p$, we find by \eqref{eq:Di1} that
$(e^{-t(p-\delta)} D_i(t))_{t\geq 0}$ is $\cL^1$-bounded. Then,
inductively using \eqref{eq:Di2} shows that
$(e^{-t(p-\delta)} D_i(t))_{t\geq 0}$ is $\cL^r$-bounded for all
$r\geq 1$. In particular, this implies that the convergence in
\eqref{eq:proof31} also holds in $\cL^r$ for all $r\geq 1$. This gives
convergence of first moments, and \eqref{eq:212} follows by taking
$t\to\infty$ in \eqref{eq:Di1}.\\
If $\delta\geq p$, the almost sure convergence in \eqref{eq:proof31}
implies, since $D_i(t) \in \mathbb N_0$, that $D_i(\infty)=0$, so
there must be a finite hitting time $T^{D_i}$ of~0.

%\bibliographystyle{apalike}
%\bibliography{herpf20190110}

\end{document}